\documentclass[1p,leqno]{elsarticle}

\usepackage{amsmath,amsthm,amsfonts,amssymb,tabularx}
\usepackage[latin1]{inputenc}
\usepackage{graphicx}
\usepackage{color}
\usepackage{fancyhdr}
\usepackage{nameref}
\usepackage{pgf,tikz}
\usepackage{pdfsync}
\usetikzlibrary{arrows}
\usepackage{pgf,tikz}
\usepackage{hyperref}
\hypersetup{urlcolor=blue, colorlinks=true} 

\newtheorem{theorem}{Theorem}[section]

\newtheorem{definition}[theorem]{Definition}

\newtheorem{remark}[theorem]{Remark}

\newcommand{\R}{\ensuremath{\mathbb{R}}}
\newcommand{\N}{\ensuremath{\mathbb{N}}}

\newcommand{\Z}{\ensuremath{\mathcal{Z}}}
\newcommand{\Levy}{\ensuremath{\mathcal{L}}}

\newcommand{\Operator}{\ensuremath{\mathfrak{L}^{\sigma,\mu}}}

\newcommand{\veps}{\varepsilon}

\newcommand{\uu}{\hat{u}}
\newcommand{\vv}{\overline{v}}
\newcommand{\B}{B_\varepsilon^{\sigma,\mu}}

\newcommand{\W}{\mathcal{W}}

\newcommand{\U}{\mathcal{U}}

\newcommand{\dd}{\,\mathrm{d}}

\newcommand{\dell}{\partial}
\newcommand{\indikator}{\mathbf{1}_{|z|\leq 1}}

\newcommand{\Lsig}{L^\sigma}
\newcommand{\Lmu}{\mathcal{L}^\mu}

\newcommand{\dersigma}{\partial_{\sigma_i}}

\DeclareMathOperator*{\esslim}{ess\,lim}

\DeclareMathOperator{\supp}{supp}

\numberwithin{equation}{section}

\begin{document}

\begin{frontmatter}

\textbf{\begin{center}
Partial differential equations
\end{center}}

\title{\bf 
On distributional solutions of local and nonlocal problems of porous medium type
}
\author{F\'elix del Teso}
\ead{felix.delteso@ntnu.no}
\author{J\o rgen Endal}
\ead{jorgen.endal@ntnu.no}
\author{Espen R. Jakobsen\corref{cor1}}
\ead{espen.jakobsen@ntnu.no}
\address{NTNU Norwegian University of Science and Technology, NO-7491 Trondheim, Norway
}
\cortext[cor1]{Corresponding author.}
\begin{abstract}
We present a theory of well-posedness and a priori estimates for
bounded distributional (or very weak) solutions of
\begin{align}\label{E0}
\dell_tu-\Operator[\varphi(u)]&=g(x,t) &&\text{in}\quad \R^N\times(0,T),
\end{align}
where $\varphi$ is merely continuous and nondecreasing and $\Operator$ is
the generator of a general symmetric L\'evy process. This means that
$\Operator$  
can have both local and nonlocal parts like
e.g. $\Operator=\Delta-(-\Delta)^{\frac12}$. New uniqueness  
results for bounded distributional solutions of this problem and
the corresponding elliptic equation are presented and proven. A
key role is played by a new Liouville type result
for $\Operator$. Existence and a priori estimates are
deduced from a numerical approximation,
and energy type estimates are also obtained.
\medskip

\noindent \textbf{R\'{e}sum\'e}

\textbf{Sur des solutions distributionelles de problèmes locaux et
  non locaux de type milieux poreux.} 
Nous montrons l'unicit\'e, l'existence, et des estimations a priori pour
des solutions distributionelles born\'ees de \eqref{E0}, o\`u
  $\varphi$ est continue et croissante et $\Operator$ est le 
g\'en\'erateur d'un processus de L\'evy sym\'etrique g\'en\'eral. Cel\`a veut dire que
$\Operator$ peut avoir des parties locales et non locales, comme par exemple
 $\Operator=\Delta-(-\Delta)^{\frac12}$. Nous pr\'esentons et montrons des
 nouveaux r\'esultats d'unicit\'e pour des solutions distributionelles
 born\'ees de ce probl\`eme. 
Un nouveau r\'esultat de type Liouville pour $\Operator$ joue un r\^ole cl\'e.
L'existence et des estimations a priori sont d\'eduites d'une approximation
num\'erique; des in\'egalit\'es de type \'energie sont aussi obtenues.
\end{abstract}

\begin{keyword} distributional solutions, uniqueness, existence, a priori estimates, energy estimates, parabolic and elliptic problems, local and nonlocal operators, Laplacian, fractional Laplacian
\MSC 35K55, 
35K65, 	
35A01, 
35R11  


\end{keyword}

\end{frontmatter}

\noindent \textbf{Version fran\c{c}aise abr\'{e}g\'e{e}}

Nous \'etudions le probl\`eme de Cauchy pour l'\'equation de diffusion
non lin\'eaire de type L\'evy \eqref{eq:mainEq}. Ici $u$ est la solution,
$u_0$ la donn\'ee initiale, $\varphi:\R\to\R$ une fonction continue
croissante quelconque, $g$ le terme du membre de droite de l'\'equation, et
$T>0$. L'op\'erateur de diffusion $\Operator$ est d\'efini par
\eqref{eq:genOp}, \eqref{eq:locOp} et \eqref{eq:levOp}, et pourrait
\^etre le g\'en\'erateur d'un processus de L\'evy quelconque comme le
Laplacien ou Laplacian fractionaire.

Dans cette note, nous donnons des r\'esultats d'existence, d'unicit\'e, et
des estimations a priori pour les solutions distributionelles de
\eqref{eq:mainEq}--\eqref{eq:inCond} dans $L^1\cap
L^\infty$, ainsi que pour
son \'equation elliptique associ{\'e}e \eqref{eq:ellipEq}. Les preuves sont li\'ee \`a
l'article \cite{BrCr79} et \`a des extentions r\'ecentes de
\cite{DTEnJa17a}.

Les r\'esultats d'unicit\'e de la premi{\`e}re partie de cette note
jouent un r\^ole cl\'e dans les preuves de convergence des m\'ethodes
num\'eriques de \cite{DTEnJa17d}. Dans la deuxi{\`e}me partie, nous annon\c{c}ons quelques r\'esultats de \cite{DTEnJa17d}. Nous
obtenons l'existence des solutions distributionelles via une
approximation num\'erique de
\eqref{eq:mainEq}--\eqref{eq:inCond}, ainsi qu'un principe de contraction dans
$L^1$, un principe de comparaison, la d\'ecroissance des normes  $L^1$ et
$L^\infty$, et la continuit\'e en temps pour la norme $L^1$. Ensuite, d'apr\`es
les r\'esultats de \cite{DTEnJa17b}, nous h\'eritons d'une famillie
d'inegalit\'ees d'\'energie, ce qui implique en particulier la
d\'ecroissance des normes $L^p$ pour chaque $1<p<\infty$.

\section{Introduction}
We study the Cauchy problem for the nonlinear L\'evy type diffusion
equation
\begin{align}\label{eq:mainEq}
\dell_t
u-\Operator[\varphi(u)]&=g(x,t) &&\text{in}\quad Q_T:=\R^N\times(0,T),\\
u(x,0)&=u_0(x) &&\text{on}\quad \R^N, \label{eq:inCond}
\end{align}
where $u=u(x,t)$ is the solution, $u_0$ the initial data,
$\varphi:\R\to\R$ an arbitrary continuous nondecreasing function, $g$
the right-hand side, and $T>0$. For smooth
functions $\psi$, the diffusion operator $\Operator$ is defined~as 
\begin{equation}\label{eq:genOp}
\Operator[\psi]:=\Lsig[\psi]+\Lmu[\psi],
\end{equation}
where the local and nonlocal parts are given by
\begin{align}
  \label{eq:locOp}
\Lsig[\psi](x)&:=\text{tr}\big(\sigma\sigma^TD^2\psi(x)\big)=\sum_{i=1}^P
\dersigma^2\psi(x)  \qquad \textup{where} \qquad \dersigma:=
\sigma_i\cdot D ,\\\label{eq:levOp}
\Levy ^\mu [\psi](x)&:=\int_{\R^N\setminus\{0\} }
\big(\psi(x+z)-\psi(x)-z\cdot D\psi(x)\indikator\big) \dd\mu(z),
\end{align}
and  $\sigma=(\sigma_1,....,\sigma_P)\in\R^{N\times P}$, $P\in \N$ and
$\sigma_i\in \R^N$, and  $\mu$ are nonnegative symmetric Radon measures.
This class of diffusion operators coincides with the class generators of
symmetric L\'evy processes. Examples are the classical Laplacian
$\Delta$, fractional Laplacians 
$(-\Delta)^\frac{\alpha}{2}$ with $\alpha\in(0,2)$, relativistic
Schr\"odinger type operators
$m^{\alpha}I-(m^2I-\Delta)^{\frac{\alpha}{2}}$ with $\alpha\in(0,2)$ and $m>0$, strongly degenerate
operators, and, surprisingly, numerical discretizations of
$\Operator$. Due to the general assumptions on 
$\varphi$, (generalized) porous medium, fast diffusion, and Stefan
type problems are included in \eqref{eq:mainEq}--\eqref{eq:inCond}. 

In this note we present new existence and uniqueness results and a
priori estimates for distributional
solutions of \eqref{eq:mainEq}--\eqref{eq:inCond} in $L^1\cap
L^\infty$. In particular, we 
present and prove new uniqueness results for bounded distributional
solutions of both \eqref{eq:mainEq}--\eqref{eq:inCond} and the related elliptic equation
\begin{align}\label{eq:ellipEq}
w-\Operator[\varphi(w)]&=f(x) &&\textup{on} \quad \R^N.
\end{align} 
The proofs are inspired by the seminal
 work \cite{BrCr79} and the later extension to the nonlocal setting in \cite{DTEnJa17a}. Most of the other properties generalize well-known results both for the local case $\Operator =\Delta$
(cf. \cite{Vaz07}) and for the nonlocal case
$\Operator=-(-\Delta)^{\frac{\alpha}{2}}$ with $\alpha\in(0,2)$
(cf. \cite{DPQuRoVa12}).

 These uniqueness results will play a
 crucial role in the convergence proofs for numerical methods in
 \cite{DTEnJa17d}. In  this note we also
 announce some of the results of \cite{DTEnJa17d}. From  
 a novel numerical approximation of
 \eqref{eq:mainEq}--\eqref{eq:inCond} we obtain existence of distributional
 solutions, $L^1$ contraction, comparison principle, decay of the $L^1$
 and $L^\infty$ norms, and continuity in time of the
 $L^1$ norm. Moreover, by adapting the results of \cite{DTEnJa17b} we
 also inherit a 
 family of energy estimates which, in particular, allow us to show
 decay of any $L^p$ norm for $1<p<\infty$.

\section{Main results}

We use the following assumptions:
\begin{align}
&\varphi:\R\to\R\text{ is nondecreasing and continuous}.
\tag{$\textup{A}_\varphi$}
\label{phias}\\
&g\in L^1(Q_T)\cap L^1(0,T;L^\infty(\R^N)).
\tag{$\textup{A}_g$}
\label{gas}\\
&u_0\in L^1(\R^N)\cap L^{\infty}(\mathbb{R}^N).
\tag{$\textup{A}_{u_0}$}
\label{u_0as}\\
&\label{muas}\tag{$\textup{A}_{\mu}$} \mu \text{ is a nonnegative symmetric Radon measure on
}\R^N\setminus\{0\}
\text{ satisfying  }
  \textup{$\textstyle\int_{|z|>0}$}\min\{|z|^2,1\}\dd
  \mu(z)<\infty\nonumber.
\end{align}

The notation $(f,g):=\int_{\R^N} fg \dd x$ is used whenever the integral is well-defined. If $f,g\in L^2$, we write $(f,g)_{L^2}$.

\begin{definition} Let $u_0\in L^1_{\textup{loc}}(\R^N)$ and $g\in
  L^1_{\textup{loc}}(Q_T)$. We say
  that function $u\in L^\infty(Q_T)$ is a distributional (or very weak) solution of \eqref{eq:mainEq}--\eqref{eq:inCond} if 
\begin{equation}\label{D-soln}
\int_0^T\int_{\R^N} \big(u\dell_t\psi + \varphi(u)\Operator[\psi] +g \psi\big)\dd x\dd t=0 \qquad \textup{for all} \qquad \psi\in C_\textup{c}^\infty(Q_T),
\end{equation}
and $\esslim_{t\to0^+} \int_{\R^N} u(x,t)\psi(x,t)\dd x=\int_{\R^N}u_0(x)\psi(x,0)\dd x$ for all $\psi\in C_\textup{c}^\infty(\R^N \times [0,T))$.
\end{definition}
Under our assumptions $\|\Operator[\psi]\|_{L^1}\leq
  C\|\psi\|_{W^{2,1}}$, see Lemma 3.5 in \cite{DTEnJa17a}, so \eqref{D-soln} is well-defined for
  $u\in L^\infty$.
\begin{remark}
\begin{enumerate}[{\rm (a)}]
\item Associated to the operator $\Operator$ is a bilinear form
  defining an energy: for $\phi, \psi\in C_\textup{c}^\infty(\R^N)$,
  $\mathcal{E}_{\sigma,\mu}[\phi,\psi]:=-(\phi, \Operator[\psi])$. Equivalently
  (cf. \cite[Section 4]{DTEnJa17b}),
\begin{equation*}
\mathcal{E}_{\sigma,\mu}[\phi,\psi]=\sum_{i=1}^P \int_{\R^N}\dersigma\phi(x)\dersigma\psi(x)\dd x+\frac12\int_{\R^N} \int_{|z|>0} \left(\phi(x+z)-\phi(x)\right)(\psi(x+z)-\psi(x)) \dd\mu(z) \dd x.
\end{equation*}
The energy of a function $\phi$ is then defined as $ \overline{\mathcal{E}}_{\sigma,\mu}[\phi] :=\mathcal{E}_{\sigma,\mu}[\phi,\phi]$.
\item $\Operator$ is a Fourier multiplier operator,  $\mathcal{F}(\Operator[\psi])(\xi)=-\widehat{\mathfrak{L}}^{\sigma,\mu}(\xi)\mathcal{\mathcal{F}(\psi)(\xi)}$, where
\begin{equation*}
\widehat{\mathfrak{L}}^{\sigma,\mu}(\xi):=\widehat{L}^{\sigma}(\xi)+\widehat{\mathcal{L}}^{\mu}(\xi)= \sum_{i=1}^P (\sigma_i\cdot \xi)^2+\int_{|z|>0} \left(1- \cos(z\cdot \xi)\right)\dd\mu(z) .
\end{equation*}
The square root operator $(\Operator)^\frac{1}{2}$ is defined as the operator with Fourier symbol $-(\widehat{\mathfrak{L}}^{\sigma,\mu}(\xi))^\frac{1}{2}$.
\end{enumerate}
\end{remark}

\begin{theorem}[Well-posedness]\label{thm:main}
Assume \eqref{phias}, \eqref{gas}, \eqref{u_0as}, and \eqref{muas}.

\noindent\textup{(a)} There exists a unique distributional solution $u\in L^1(Q_T)\cap L^\infty(Q_T)\cap C([0,T];L_\textup{loc}^1(\R^N))$ of \eqref{eq:mainEq}--\eqref{eq:inCond}. 

\noindent\textup{(b)}  If $u, v$ are solutions with data $u_0,v_0$ and $g,h$ satisfying resp. \eqref{u_0as} and \eqref{gas}, then, for every $t\in[0,T]$,
\begin{enumerate}[(i)]
\item \textup{(\textup{$L^1$ contraction})}
$
\int_{\R^N}(u(x,t)-v(x,t))^+\dd x\leq \int_{\R^N}(u_{0}(x)-v_{0}(x))^+\dd x+\int_0^t\int_{\R^N}(g(x,\tau)-h(x,\tau))^+\dd x\dd \tau;
$
\item \textup{(\textup{Comparison})} if $u_{0}\leq v_0$ a.e. and $g\leq h$ a.e.,  then $u\leq v$ a.e.;
\item \textup{(\textup{$L^p$ estimate 1})} for $1\leq p\leq\infty$,
$\|u(\cdot,t)\|_{L^p(\R^N)}\leq\|u_0\|_{L^p(\R^N)}+\int_0^t\|g(\cdot,\tau)\|_{L^p(\R^N)}\dd \tau$;
\item \textup{($L^p$ estimate 2)} for $1< p< \infty$, $\|u(\cdot,t)\|^p_{L^p(\R^N)}\leq\|u_0\|^p_{L^p(\R^N)}+p\int_0^t\int_{\R^N} |u(x,\tau)|^{p-2} u(x,\tau) g(x,\tau) \dd x \dd t;$ 
\item \textup{(Energy estimate)} if  $\Phi:\R\to\R$ is defined by $\Phi(\xi):=\int_0^\xi \varphi(\eta)\dd \eta$, then
\begin{equation*}
\int_{\R^N}\Phi(u(x,t))\dd x \ +\int_0^t \overline{\mathcal{E}}_{\sigma,\mu}[\varphi(u(\cdot,\tau))]\dd\tau  \leq \int_{\R^N}\Phi(u_0(x))\dd x+\int_0^t\int_{\R^N}g(x,\tau)\varphi(u(x,\tau))\dd x\dd \tau;
\end{equation*}
\item \textup{(\textup{Time regularity})} for every $t, s\in[0,T]$ and every compact set $K\subset \R^N$, 
\begin{equation*}
\|u(\cdot,t)-u(\cdot,s)\|_{L^1(K)}\leq 2\lambda\big(|t-s|^{\frac{1}{3}}\big)+C\big(|t-s|^{\frac{1}{3}}+|t-s|\big) +|K|\int_s^t\|g(\cdot,\tau)\|_{L^\infty(\R^N)}\dd \tau,
\end{equation*}
where
$\lambda(\delta)=\max_{|h|\leq\delta}\|u_0-u_0(\cdot+h)\|_{L^1(\R^N)}$ and $C=C(K,u_0,\varphi)>0$;
\item \textup{(Conservation of mass)} if, in addition, there exist $L,\delta>0$ such that $|\varphi(r)|\leq L|r|$ for $|r|\leq \delta$, then
\begin{equation*}
\int_{\R^N}u(x,t)\dd x= \int_{\R^N}u_0(x)\dd x+ \int_0^t\int_{\R^N} g(x,\tau)\dd x\dd \tau.
\end{equation*}
\end{enumerate}
\end{theorem}


\section{Uniqueness of distributional solutions}
We obtain uniqueness for a class of bounded distributional solutions
of \eqref{eq:mainEq}--\eqref{eq:inCond} and \eqref{eq:ellipEq}. One of
the key tools in the proof of these results is the Liouville type
result given by Theorem \ref{localnonlocalLiouville}. 

\begin{theorem}[Uniqueness 1]\label{thm:uniquePP}
Assume \eqref{phias}, \eqref{muas}, $g\in L^1_{\textup{loc}}(Q_T)$, and $u_0\in L^\infty(\R^N)$. Then there is at most one distributional solution $u$ of \eqref{eq:mainEq}--\eqref{eq:inCond} such that $u \in L^\infty(Q_T)$ and $u-u_0\in L^1(Q_T)$.
\end{theorem}

\begin{theorem}[Uniqueness 2]\label{thm:uniqueEP}
Assume \eqref{phias}, \eqref{muas}, and $f\in L^\infty(\R^N)$. Then there is at most one distributional solution $w$ of \eqref{eq:ellipEq} such that $w \in L^\infty(\R^N)$ and $w-f\in L^1(\R^N)$.
\end{theorem}

\begin{theorem}[``Liouville'']\label{localnonlocalLiouville}
Assume \eqref{muas} and that either $\sigma\not\equiv0$ or $\supp \mu\not=\emptyset$. If $v\in C_0(\R^N)$ solves $
\Operator[v] = 0$ in $\mathcal{D}'(\R^N),$ then $v\equiv0$ in $\R^N$.
\end{theorem}

\begin{proof}
If $\sigma\equiv0$, then $\Operator=\Levy^\mu$ and the result follows
by Theorem 3.9 in \cite{DTEnJa17a}. Assume that
$\sigma\not\equiv0$, and note that by a change of coordinates we may also assume
that $L^\sigma=\Delta_l:=\sum_{i=1}^l \dell_{x_i}^2$ for some $1\leq
l\leq N$.

Let $\omega_\delta$ be a standard mollifier in $\R^N$ and define $v_\delta:=v*\omega_\delta \in C_0(\R^N)\cap C_\textup{b}^\infty(\R^N)$. As shown in the proof of Theorem 3.9 in \cite{DTEnJa17a}, 
$
\int_{\R^N}v(y) \Levy^\mu[\omega_\delta(x-\cdot)](y)\dd y=\Levy^\mu[v_\delta](x).
$
We also have that 
$
\int_{\R^N}v(y)\Delta_l[\omega_\delta(x-\cdot)](y)\dd y=\Delta_l[v_\delta](x).
$
In this way, taking $w_\delta(x-y)$ as a test function in the distributional formulation we get that
\begin{equation}\label{eq:liouReg}
\Delta_l[v_\delta](x)+\Levy^\mu[v_\delta](x)=0 \qquad \textup{for every} \qquad x\in \R^N.
\end{equation}
Now we multiply \eqref{eq:liouReg} by $v_\delta$, integrate over $\R^N$, integrate by parts, and use Plancherel's theorem to get
\begin{equation*}
\begin{split}
0=-\sum_{i=1}^l\int_{\R^N}v_\delta(x)\dell_{x_i}^2 v_\delta(x)\dd x-\int_{\R^N}v_\delta(x)\Levy^\mu[v_\delta](x)\dd x=\sum_{i=1}^l\int_{\R^N}\left|\dell_{x_i} v_\delta(x)\right|^2\dd x+\|(\Lmu)^{\frac{1}{2}}[v_\delta]\|_{L^2(\R^N)}^2.
\end{split}
\end{equation*}
Since all the terms in the last expression are nonnegative, they are all zero. In particular $\int_{\R^N} \left|\partial_{x_1} v_\delta(x)\right|^2\dd
x=0$, and then $\partial_{x_1} v_\delta(x)=0$ for every $x\in \R^N$. Hence
$0=\int_{x_1}^b\partial_{x_1} v_\delta(s,x')\dd s=v_\delta(b,x')-v_\delta(x_1,x')$ for every $x_1<b$ and every $x'=(x_2,\cdots,x_N)\in \R^{N-1}.$
Since $v_\delta \in C_0(\R^N)$, we send $b\to\infty$ in the
previous expression to see that $v_\delta(x_1,x')=v_\delta(b,x')\to0$ as
$b\to\infty$. Hence $v_\delta(x)=0$ for every $x\in
\R^N$. By properties of mollifiers, $v_\delta\to v$ locally uniformly
in $\R^N$ as $\delta\to0^+$, which means that also $v(x)=0$ for every $x\in \R^N$. 
\end{proof}

\begin{proof}[Proof of Theorem \ref{thm:uniquePP}]
\noindent\textsc{Step 1: The resolvent $\B$ of $\Operator$.} Formally
the resolvent of 
$\Operator$ is given as $\B=(\veps I-\Operator)^{-1}$ for
$\veps>0$. But to give a rigorous meaning to this operator even
when $\Operator$ is strongly degenerate, we define it as
$\B[\gamma](x):=v_\veps(x)$ where $v_\veps$ is the solution of the
linear elliptic equation 
\begin{equation}\label{eq:linEllip}
\veps v_\veps(x)-\Operator[v_\veps](x)=\gamma(x) \qquad\text{in}\qquad \R^N.
\end{equation}
To be able to apply $\B$ to $L^1$, $L^\infty$, and smooth $\gamma$, we
need to prove existence and uniqueness for $L^1$ and $L^\infty$
distributional and $C_\textup{b}^\infty$ classical solutions of 
\eqref{eq:linEllip} along with the following estimates
\begin{equation}\label{eq:estEE}
\veps\|\B[\gamma]\|_{L^1}\leq \|\gamma\|_{L^1}, \quad
\veps\|\B[\gamma]\|_{L^\infty}\leq \|\gamma\|_{L^\infty}, \quad
\textup{and} \quad \veps\|D^\beta\B[\gamma]\|_{L^\infty}\leq \|D^\beta
\gamma\|_{L^\infty}\ \forall \beta\in\N^N .
\end{equation}
The proof can be deduced by following the ideas of the proof of
Theorem 3.1 in \cite{DTEnJa17a}. The idea is to approximate
$\Operator$ by a bounded nonlocal operator $\Levy^{\nu^h}$, and then approximate
\eqref{eq:linEllip} by the equation
\begin{equation}\label{eq:linEllipAprox}
\veps v_{h,\veps}(x)-\Levy^{\nu^h}[v_{h,\veps}](x)=\gamma(x) \qquad\text{in}\qquad \R^N.
\end{equation} 
Because of the local terms, we have to modify the choice of $\nu^h$ from
\cite{DTEnJa17a} and take
\begin{equation}\label{eq:discaprox}
\nu^h(z):=\nu^h_\sigma(z)+\nu^h_\mu(z)=\frac{1}{h^2}\sum_{i=1}^P\left(\delta_{h\sigma_i}(z)+\delta_{-h\sigma_i}(z)\right)+\mu(z)\mathbf{1}_{|z|>h},
\end{equation}
where $\delta_a$ is the delta-measure supported at $a$.  By a similar
argument as in Lemma 5.2 in \cite{DTEnJa17a}, $\nu^h$ is a nonnegative
symmetric Radon measure satisfying $\nu^h(\R^N)<\infty$ and
$\|\Levy^{\nu^h}[\psi]-\Operator[\psi]\|_{L^p(\R^N)}\to0$ as $h\to0^+$
for all $\psi \in C^\infty_\textup{c}(\R^N)$ and $p=\{1,\infty\}$.
 Note that $\Levy^{\nu^h}$ is in the class of operators
 \eqref{eq:levOp} with $\mu=\nu^h$ satisfying \eqref{muas}, and thus,
 \eqref{eq:linEllipAprox} has already been studied in
 \cite{DTEnJa17a}. In particular, we have existence, uniqueness and
 estimates \eqref{eq:estEE} for solutions 
 of \eqref{eq:linEllipAprox} by Theorem 3.1 in \cite{DTEnJa17a}. The
 corresponding results for equation \eqref{eq:linEllip} then follow
 using compactness arguments to pass to the 
 limit as $h\to0^+$ and then verifying that the limit satisfies equation
 \eqref{eq:linEllip}. There are 3 different cases, $L^1$, $L^\infty$,
 and smooth, but all arguments follow as in \cite{DTEnJa17a} with
 only easy modifications. To give an idea we do the case of smooth  
 solutions when $\gamma\in C_{\textup{b}}^\infty$ (cf. Proposition 6.12 in \cite{DTEnJa17a}). The Arzel\`a-Ascoli
 theorem and the third estimate in \eqref{eq:estEE} ensure that there
 is a function $\vv_{\veps}$ such that $(v_{h,\veps}, Dv_{h,\veps}, 
 D^2v_{h,\veps})\to (\vv_{\veps}, D\vv_{\veps}, D^2\vv_{\veps})$
 locally uniformly as $h\to0^+$. To see that $\vv_\veps$ is a
 classical solution of \eqref{eq:linEllip}, it remains to show that $\Levy^{\nu^h}[v_{h,\veps}](x)\to\Operator[\vv_\veps](x)$ in $\R^N$. Indeed,
\begin{equation*}
\begin{split}
|\Levy^{\nu^{h}}[v_{h,\veps}](x)-\Operator[\bar{v}_\veps](x)|\leq |\Levy^{\nu^{h}_\sigma}[v_{h,\veps}](x)-\Lsig[\bar{v}_\veps](x)|+|\Levy^{\nu^{h}_\mu}[v_{h,\veps}](x)-\Lmu[\bar{v}_\veps](x)|.
\end{split}
\end{equation*}
The first term on the right-hand side converges to zero as in the proof of Proposition 6.12 in \cite{DTEnJa17a}, while for the remaining one~we~have
\begin{equation*}
\begin{split}
&|\Levy^{\nu^{h}_\sigma}[v_{h,\veps}](x)-\Lsig[\bar{v}_\veps](x)|\leq |\Levy^{\nu^{h}_\sigma}[\bar{v}_\veps](x)-\Lsig[\bar{v}_\veps](x)|+|\Levy^{\nu^{h}_\sigma}[v_{h,\veps}-\vv_\veps]|\\
&\leq h^2\| D^4 \bar{v}_\veps\|_{L^\infty(\R^N)}\sum_{i=1}^P\sum_{|\alpha|=4}\frac{2}{\alpha !}|\sigma_i|^\alpha+ \sum_{i=1}^P\max_{|\xi|\leq h}|D^2(v_{h,\veps}-\bar{v}_\veps)(x+\xi\sigma_i)|\sum_{|\alpha|=2}\frac{2}{\alpha !}|\sigma_i|^\alpha.
\end{split}
\end{equation*}
This concludes the proof of existence since $D^2v_{h,\veps}\to D^2\vv_{\veps}$
locally uniformly as $h\to0^+$. Repeating the compactness argument for
higher derivatives and passing to the limit we find that
$\vv_{\veps}$ also satisfies the third estimate in
\eqref{eq:estEE}. Uniqueness is a trivial consequence of the linearity
of \eqref{eq:linEllip} and the estimates in \eqref{eq:estEE}.

\noindent\textsc{Step 2: $\veps B_{\veps}^{\sigma,\mu}[q]\to 0$ a.e. as $\veps \to0^+$ for
  $q\in L^1(\R^N)\cap L^\infty(\R^N)$.} Let $\gamma\in
C_\textup{c}^\infty(\R^N)$ and $\Gamma_\veps:=\veps
\B[\gamma]$. 
We first show that all  subsequences 
$\{\Gamma_{\veps_j}\}_{j}$ converging in $L^\infty_\textup{loc}$ as $\veps_j\to0^+$
converge to $\Gamma\equiv0$. Indeed, by~\eqref{eq:linEllip}
$$
\veps_j\int_{\R^N} \Gamma_{\veps_j} \psi \dd x- \int_{\R^N}
\Gamma_{\veps_j} \Operator [\psi]\dd x = \veps_{j} \int_{\R^N}\gamma
\psi \dd x\qquad\text{for all}\qquad \psi\in C^\infty_c(\R^N),$$
and we send $\veps_j\to0^+$ to find that
$\Operator[\Gamma]=0$ in $\mathcal{D}'$. Since $\Gamma$ is
Lipschitz and in $L^1$ by \eqref{eq:estEE},
$\lim_{|x|\to \infty}\Gamma(x)=0$, and then $\Gamma\equiv0$ by the
Liouville type result Theorem \ref{localnonlocalLiouville}. The next step is to observe that
$\Gamma_\veps$ is equibounded and equi-Lipschitz  by \eqref{eq:estEE},
and use the first part and the Arzel\`a-Ascoli theorem to conclude
that any subsequence of $\{\Gamma_\veps\}_{\veps>0}$ has a further
subsequence converging to zero in $L^\infty_{\textup{loc}}$. This
implies that the whole sequence converges  to zero in $L^\infty_{\textup{loc}}$.
Now we study $Q_\veps:=\veps B_{\veps}^{\sigma,\mu}[q]$.
By self-adjointness of $\B$
(cf. Lemma 3.4 in \cite{DTEnJa17a}), the  properties of
$\Gamma_\veps$, and the dominated convergence theorem,
$\int_{\R^N}Q_\veps \gamma \dd x=\int_{\R^N}
q\, \Gamma_{\veps} \dd x \to \int_{\R^N} q\, \Gamma \dd x=0$,
i.e., $Q_\veps\to0$ in $\mathcal{D}'$ as $\veps\to0^+$.
Then since $\mathcal D'$ and $L_\textup{loc}^1$ limits coincide and $\{Q_\veps\}_{\veps>0}$ is precompact in $L_\textup{loc}^1$
by \eqref{eq:estEE} and Kolmogorov's compactness theorem, all
subsequences of $\{Q_\veps\}_{\veps>0}$
have further subsequences converging to zero in $L_\textup{loc}^1$
and a.e. The full sequence thus converges to zero~a.e.

\noindent\textsc{Step 3: The difference $\U$ of two
  solutions~of~\eqref{eq:mainEq}--\eqref{eq:inCond} and ``energy''
  from $\B$.} Let $u,\hat{u}\in L^\infty(Q_T)$ be two distributional
solutions of \eqref{eq:mainEq}--\eqref{eq:inCond} with initial data
$u_0$ such that $u-u_0, \hat{u}-u_0\in L^1(Q_T)$. Define 
$\U:=u-\hat{u}$ and $\mathcal{Z}:=\varphi(u)-\varphi(\hat{u})\in
L^\infty(Q_T)$. Note that
$\|u-\hat{u}\|_{L^1(Q_T)}\leq\|u-u_0\|_{L^1(Q_T)}+\|u-u_0\|_{L^1(Q_T)}<\infty$,
and thus, $\U\in L^1(Q_T)\cap L^\infty(Q_T)$. We subtract the equations for $u$ and
$\uu$ (distributional formulation of  \eqref{eq:mainEq}), and take $\psi=\B[\gamma]$ for $\gamma\in C_\textup{c}^\infty(\R^N)$ as test function. By the properties of solutions of \eqref{eq:linEllip}, we get
$
\int_0^T\int_{\R^N} \big(\U\B[\dell_t\gamma]+\Z(\veps\B[\gamma]-\gamma)\big)\dd x\dd t
=0$. Thus, by the self-adjointness of $\B$,
\begin{equation}\label{eq:timederB}
\dell_t\B[\U]=\veps \B[\Z]-\Z \qquad \textup{in} \qquad \mathcal{D}'(Q_T).
\end{equation}
Now consider the ``energy'' like function $h_\veps (t)=\int_{\R^N} \B[\U](x,t)\,\U(x,t)\dd x$. Note that by \eqref{eq:estEE}, $h_\veps\in L^1(0,T)$ since $\|h_\veps\|_{L^1(0,T)}\leq \frac{1}{\veps}\|\U\|_{L^\infty(Q_T)}\|\U\|_{L^1(Q_T)}$. 
As in Proposition 3.11 in \cite{DTEnJa17a}, we get that  $h_\veps$ is absolutely continuous and $h'_\veps(t)=2\left(\dell_t\B[\U](\cdot,t), \U(\cdot,t)\right)$ in $\mathcal{D}'(0,T).$  By \eqref{eq:timederB} and \eqref{eq:posTermsZero} below, and since $\Z\U\geq0$,
\begin{equation}\label{eq:estH}
0\leq h_\veps(t) = h_\veps(0^+)+ \int_0^t h_\veps'(s)ds\leq 0+2\int_0^t \left(\veps \B[\Z](\cdot,s), \U(\cdot,s)\right) \dd s.
\end{equation}
Let now $\xi>0$. By self-adjointness of $\B$, we have for a.e $t\in[0,T]$
\begin{equation}\label{eq:L1Z}
\left(\veps \B[\Z](\cdot,t), \U(\cdot,t)\right) \leq \|\Z\|_{L^\infty(Q_T)}\int_{\R^N}\left|\veps \B[\U](x,t)\right|\mathbf{1}_{\left|\Z(x,t)\right|>\xi} \dd x + \xi\|\U(\cdot,t)\|_{L^1(\R^N)}.
\end{equation}
Note that $\left|\veps
\B[\U](x,t)\right|\mathbf{1}_{\left|\Z(x,t)\right|>\xi} \leq
\|\U\|_{L^\infty(Q_T)} \mathbf{1}_{\left|\Z(x,t)\right|>\xi} \in
L^1(\R^N)$ (see \cite{BrCr79} and also Lemma 3.13 in
\cite{DTEnJa17a}), and hence by Step 2 with $q=\U(\cdot,t)\in L^1(\R^N)\cap L^\infty(\R^N)$, the first integral on
the right-hand side of
\eqref{eq:L1Z} goes to zero as $\veps \to 0^+$. Then sending $\xi\to0^+$ in the above estimate and using  Lebesgue's dominated convergence theorem in \eqref{eq:estH}, we conclude that, up to a subsequence, $h_{\veps_j}(t)\to0$ as $\veps_j\to0^+$ for a.e $t\in[0,T]$. 

\noindent\textsc{Step 4: Deducing that $\U\equiv0$.} Since all terms
in \eqref{eq:linEllip} are in $L^2$, for a.e. $t\in[0,T]$, 
\begin{equation}\label{eq:posTermsZero}
\begin{split}
h_\veps(t)&=\left(\B[\U](\cdot,t), \veps \B[\U](\cdot,t) - \Operator[\B[\U]](\cdot,t)\right)_{L^2(\R^N)}\\
&= \veps\left\| \B[\U](\cdot,t) \right\|_{L^2(\R^N)}^2 +\|(\Operator)^{\frac{1}{2}}[\B[\U]](\cdot,t)\|_{L^2(\R^N)}^2.
\end{split}
\end{equation}
By the conclusion of Step 3 and since all terms in the last equality of \eqref{eq:posTermsZero} are nonnegative, they must all converge to zero as $\veps_j \to0^+$. Hence, the following integrals also converge to zero for all $\psi\in C_\textup{c}^\infty(\R^N)$,
\begin{equation*}
\begin{split}
&\left|\int_{\R^N}\B[\U]\Operator[\psi]\dd x\right|=\left|\int_{\R^N}(\Operator)^{\frac{1}{2}}[\B[\U]](\Operator)^{\frac{1}{2}}[\psi]\dd x\right| \leq\|(\Operator)^{\frac{1}{2}}[\B[\U]]\|_{L^2}\|(\Operator)^{\frac{1}{2}}[\psi]\|_{L^2},
\end{split}
\end{equation*}
and $\left|\int_{\R^N}\veps\B[\U]\psi\dd x\right|\leq
\|\veps\B[\U]\|_{L^2} \|\psi\|_{L^2}.$ We thus conclude the proof by noting that $\U=\veps_jB_{\veps_j}^{\sigma,\mu}[\U]-\Operator[B_{\veps_j}^{\sigma,\mu}[\U]]\to0$ in $\mathcal{D}'(\R^N)$ for a.e. $t\in[0,T]$, that is $u-\hat{u}=\U=0$ a.e. in $Q_T$.
\end{proof}

\begin{proof}[Proof of Theorem \ref{thm:uniqueEP}] Steps 1 and 2 from the proof of Theorem \ref{thm:uniquePP} are independent of the equation itself and remain true in this case since the operator is the same. Let $w,\hat{w}\in L^\infty(\R^N)$ be two distributional solutions of \eqref{eq:ellipEq} with right-hand side $f$ such that both $w-f$ and $\hat{w}-f$ belong to $L^1(\R^N)$. Define $\W:=w-\hat{w} \in L^1(\R^N)\cap L^\infty(\R^N)$ and $\Z:=\varphi(w)-\varphi(\hat{w})\in L^\infty(\R^N)$. As before, we also define the quantity $h_\veps=(\W,\B[\W])$. Since $w$ and $\hat{w}$ are distributional solutions of \eqref{eq:ellipEq}, we have that (see Step 3 in the proof of Theorem \ref{thm:uniquePP})
\begin{equation}\label{hvepswithgamma}
\int_{\R^N}\W\B[\gamma]\dd x=\int_{\R^N}\Z\Operator[\B[\gamma]]\dd x=\int_{\R^N}\Z\big(\veps\B[\gamma]-\gamma\big)\dd x \qquad \textup{for all} \qquad \gamma\in C_\textup{c}^\infty(\R^N) .
\end{equation} 
In fact, $\gamma$ can be replaced by $\W$ in \eqref{hvepswithgamma} by
the density of $C_\textup{c}^\infty(\R^N)$ in $L^1(\R^N)$ and the
estimate\linebreak
$\veps\|\B[\gamma]-\B[\W]\|_{L^1(\R^N)}=\veps\|\B[\gamma-\W]\|_{L^1(\R^N)}\leq
\|\gamma-\W\|_{L^1(\R^N)}$. Then $h_\veps=\int_{\R^N}\Z\left(\veps\B[\W]-\W\right)\dd x$ goes to zero as $\veps \to 0^+$ like in \eqref{eq:L1Z}.  The rest of the proof follows as in the proof of Theorem \ref{thm:uniquePP} by replacing $\U$ by $\W$ and dropping the $t$ dependence of $h_\veps$.
\end{proof}

\section{Ideas on how to prove Theorem \ref{thm:main}}
\subsection{Existence and a priori estimates via numerical approximations}
Once the uniqueness given by Theorem \ref{thm:uniquePP} is available, it is possible to provide \eqref{eq:mainEq}--\eqref{eq:inCond} with existence and suitable a priori estimates for initial data $u_0\in L^1(\R^N)\cap L^\infty(\R^N)$ -- see (a), (b)(i), (b)(ii), (b)(iii) with $p=\{1,\infty\}$, and (b)(vi) of Theorem \ref{thm:main}. This task is one of the objectives of \cite{DTEnJa17d}. A crucial idea is the fact that the class of operators given by \eqref{eq:levOp} with $\mu$ satisfying \eqref{muas} is so general that it includes many monotone discretizations of the more general operator $\Operator$. In this way, we can formulate a numerical method for \eqref{eq:mainEq}--\eqref{eq:inCond}: Choose $x_\beta=h\beta, t_j=kj$ for $\beta\in\mathbb{Z}^N$, $j\in\N$, and $h,k >0$, and consider 
\begin{equation}\label{defNumSch}
  U^j(x_\beta)=U^{j-1}(x_\beta)+k\Big(\Levy^{\nu_1^h}[\varphi(U^j)](x_\beta)+\Levy^{\nu_2^h}[\varphi_2^h(U^{j-1})](x_\beta)+G^j(x_\beta)\Big),
\end{equation}
where $\Levy^{\nu_1^h}$ and $\Levy^{\nu_2^h}$ are discretizations of $\Operator$, $\nu_1^h(\R^N), \nu_2^h(\R^N)<\infty$, $\varphi_2^h$ approximate $\varphi$, $G^j$ is a time average of $g$, and $U^0$ is defined as a space average of $u_0$. In fact, if we extend \eqref{defNumSch} to all $\R^N$, the numerical method can be seen, at every time step, as a nonlinear and nonlocal elliptic equation of the form \eqref{eq:ellipEq} with 
$
w=U^j$, $\Operator = k \Levy^{\nu_1^h}$ and $f=U^{j-1}+k\big(\Levy^{\nu_2^h}[\varphi_2^h(U^{j-1})]+G^j\big)$.
In this way, we can study the properties of the numerical scheme
\eqref{defNumSch} by studying the nonlinear equation
\eqref{eq:ellipEq} and iterating in time. This leads to the
corresponding discrete time version of the above mentioned
estimates. Since approximation, stability and compactness will be used
to deduce such results, uniqueness of distributional solutions of
\eqref{eq:ellipEq} -- that is, Theorem  \ref{thm:uniqueEP} -- plays a
crucial role. By passing to the limit (up to subsequences) as $h, k
\to 0^+$, we get the continuous time estimates and also existence of
$L^1(\R^N)\cap L^\infty(\R^N)$ distributional solutions of the
parabolic problem.  
Furthermore, the uniqueness result given by Theorem \ref{thm:uniquePP} ensures that the full sequence of numerical solutions converges to the unique distributional solution of \eqref{eq:mainEq}--\eqref{eq:inCond}.

\subsection{Energy estimates and conservation of mass}
A trivial adaptation of the results and proofs presented by Corollary
2.18 and Theorems 2.19 and 2.21 in \cite{DTEnJa17b} (where the
case $\Operator =\Lmu$ is covered) shows that for solutions $u\in
L^1(Q_T)\cap L^\infty(Q_T)\cap C([0,T];L_\textup{loc}^1(\R^N))$ of
\eqref{eq:mainEq}--\eqref{eq:inCond} the concepts of distributional and
energy solutions are equivalent, and the estimates (b)(iv) and (b)(v)
of Theorem \ref{thm:main} hold. As a consequence of Theorem
\ref{thm:main} (b)(iv), we also obtain (b)(iii) with $p\in(1,\infty)$
by H\"older and Gr\"onwall inequalities. In the present setting, we must ensure the convergence of the local part of the energy, which is done using the discretization \eqref{eq:discaprox}, summation by parts, and Theorem \ref{thm:main} (b)(v):
\begin{equation*}
\overline{\mathcal{E}}_{0,\nu_\sigma^h}[\varphi(u(\cdot,t))]=-\int_{\R^N}\varphi(u) \Levy^{\nu_\sigma^h}[\varphi(u)] \dd x=\sum_{i=1}^P\int_{\R^N}\left| \frac{\varphi(u(x+h\sigma_i,t))-\varphi(u(x,t))}{h}\right|^2\dd x\leq K,
\end{equation*}
where $K=K(\varphi, u_0, g)$ is a constant. Since
the difference quotients of $\varphi(u)$ are uniformly bounded, 
the weak derivative $\dersigma\varphi(u)$ exists in $L^2$, and a standard
argument (like in Section 4 in \cite{DTEnJa17b}) shows the convergence
of the local part of the energy. To conclude, we obtain conservation
of mass by following the proof of Theorem 2.10 in
\cite{DTEnJa17a}. Note that neither the local term nor the right-hand
side $g$ add any extra difficulty to the proof. See Remark 2.11 in
\cite{DTEnJa17a} for the optimality of the condition on $\varphi$.

{\sc Acknowledgments.} 
F. ~del Teso and E.~R.~Jakobsen were supported by the Toppforsk (research excellence) project Waves and Nonlinear Phenomena (WaNP), grant no. 250070 from the Research Council of Norway. F.~del Teso was also supported by  the ERCIM  ``Alain Bensoussan'' Fellowship programme. We also thank Boris Andreianov for useful comments on the proof of Theorem~\ref{thm:uniquePP}.



\end{document}